\documentclass[11pt]{amsart}

\usepackage{amsthm}
\usepackage{amsmath}
\usepackage{amssymb}
\usepackage[dvipsnames]{xcolor}

\newcommand{\spn}{\mathop\mathrm{span}}
\newcommand{\rank}{\mathop\mathrm{rank}}

\newtheorem{thm}{Theorem}[section]
\newtheorem{theorem}[thm]{Theorem}

\newtheorem{corollary}[thm]{Corollary}

\newtheorem{example}[thm]{Example}

\newtheorem{lemma}[thm]{Lemma}
\newtheorem{proposition}[thm]{Proposition}

\newtheorem{definition}[thm]{Definition}
\theoremstyle{remark}
\newtheorem{remark}[thm]{Remark}

\newtheorem{ex}[thm]{Example}

\newcommand{\RR}{\mathbb R}
\newcommand{\NN}{\mathbb N}

\newcommand{\R}{\mathbb R}

\numberwithin{equation}{section}

\newcommand{\F}{\mathcal X}

\begin{document}

	\title{ Piecewise scalable frames}
	\author[Casazza, De Carli, Tran]{Peter G. Casazza, Laura De Carli, Tin T. Tran}
	\address{Casazza: Department of Mathematics, University
of Missouri, Columbia, MO 65211-4100.}
\address{De Carli/Tran: Department of Mathematics and Statistics, Florida International University, Miami, FL 33199.}

	
	\thanks{The first author was supported by
		NSF DMS 1609760}
	
	\email{casazzapeter40@gmail.com }
	\email{decarlil@fiu.edu}
\email{ttran@fiu.edu}
	\begin{abstract}
	 In this paper we define ``piecewise scalable frames". This new scaling process allows us to alter many frames to Parseval frames which is impossible by the previous standard scaling. We give necessary and sufficient conditions for a frame to be piecewise scalable. We show that piecewise scalability is preserved under unitary transformations. Unlike standard scaling, we show that all frames in $\RR^2$ and $\RR^3$ are piecewise scalable. We also show that if the frame vectors are close to each other, then they might not be piecewise scalable. Several properties of scaling constants are also presented. 

\vspace{0.2 in}

	\end{abstract}

	\maketitle
	
	\section{Introduction}
	
Hilbert space frame theory has become one of the most broadly applied subjects in Applied Mathematics today \cite{CG, CL, C, W}. This subject
 originated in the work of Duffin and 
Schaeffer \cite{DS} when they were studying non-harmonic Fourier series. 

\begin{definition}
A family of vectors $\{x_i\}_{i=1}^m$ in an  $n$-dimensional Hilbert space $\mathbb{H}^n$ is a {\bf frame} if there are constants $0<A\le B<\infty$ satisfying:
\[ A\|x\|^2 \le \sum_{i=1}^m|\langle x,x_i\rangle|^2 \le B\|x\|^2,\mbox{ for all }x\in \mathbb{H}^n.\]
\end{definition}
If $A=B$ this is an {\it A-tight frame} and if $A=B=1$ this is a {\it Parseval frame}.  The largest $A$ and smallest $B$ satisfying
this inequality are called the {\it lower} (respectively, {\it upper}) {\it frame bound}.  The frame is said to be {\it unit-norm} if all elements have norm one. The {\it analysis operator} of the frame is the
operator $T: \mathbb{H}^n\rightarrow \ell_2(m)$ given by $Tx= \{\langle x,
 x_i\rangle\}_{i=1}^m$. The {\it synthesis operator} of the frame
is $T^*:\ell_2(m)\rightarrow \mathbb{H}^n$ and satisfies $T^*(\{a_i\}_{i=1}^m)=\sum_{i=1}^ma_ix_i$.  The {\it frame operator} is $S=T^*T$ and is the
positive, self-adjoint, invertible operator $S: \mathbb{H}^n\rightarrow \mathbb{H}^n$ given by:
\[ S(x)=\sum_{i=1}^m\langle x, x_i\rangle x_i.\]
$B$ is the largest eigenvalue of $S$ and $A$ is the smallest eigenvalue of $S$.
The quotient $B/A$ is called the {\it condition number}. 

We recover vectors by the formula:

\[ x=S^{-1}S(x)=\sum_{i=1}^m\langle x ,x_i\rangle S^{-1}x_i=\sum_{i=1}^m\langle x,S^{-1/2}x_i\rangle S^{-1/2}x_i.\]
It follows that$\{S^{-1/2}x_i\}_{i=1}^m$ is a Parseval frame.
This requires inverting the frame operator which might be difficult. Since  a frame is $A$-tight if and only if 
$Sx= Ax$ for all $x \in \mathbb{H}^n$, Parseval frames are the most desirable since $S=I$. 
So we want
to alter a frame in a simple manner to make it Parseval. 
\begin{definition}
  We say that a frame $\{x_i\}_{i=1}^m\subset \R^n$   is {\bf scalable} if there exist  constants $c_1, c_2,\ldots, c_m\in \RR$ (the {\it scaling constants} of the frame) for which 
  	$\{c_ix_i\}_{i=1}^m$ is a Parseval frame for $\R^n$.
	\end{definition}
  That is, a frame $\F$ is scalable with constants $\{c_i\}_{i=1}^m$  if, for every $x\in\R^n$, we have that 
  $$ \sum_{i=1}^m |\langle x, c_ix_i\rangle|^2 = \|x\|^2, $$ 
  where  $\langle \cdot , \cdot \rangle $ and  $\|\cdot\|  $ denote the standard scalar product and norm in $\R^n$.
  We can let $X=[x_i]_{i=1}^m$  be the matrix whose columns are the vectors $x_i$, and  observe that a frame is scalable if and only if there exists a diagonal matrix $D$ for which  the columns of the matrix $XD$ form a Parseval frame.

  There is a lot of literature on the subject of scalable frames \cite{CC,CCH,CDT,CX, CK,  CKO,CKL,DK, KO, KOP}.
  Unfortunately the frames for which  this standard scaling process is possible are very few. 
  Also, these papers scale the frame by using a very large number of $c_i=0$, which make this process unusable in practice where all the frame vectors
  are needed for the application.
  In this paper, we give a more general definition of scaling  
  and we  study the properties of frames that can be transformed into Parseval frames  using this definition. 

 \begin{definition}\label{D1}
		  A frame $\F=\{x_i\}_{i=1}^m\subset \R^n$ is {\bf   piecewise
		  	scalable}  if  there exist an orthogonal  projection $P:\R^n\to\R^n$ and   constants $\{a_i, b_i\}_{i=1}^m$
			 so that $\{a_iPx_i+b_i(I -P)x_i\}_{i=1}^m$ is a Parseval frame for $\R^n$.  
	\end{definition} 
 
 Here, $I=I_n$ denotes the identity operator in $\R^n$. When there is no ambiguity, we  also denote with  $I$   the $n\times n$ identity matrix.   We recall that an  orthogonal  projection is a linear operator $P:\R^n\to\R^n$  such that $P\circ P=P $ and $P^*=P$. If a frame is piecewise scalable with an orthogonal projection $P$, we sometimes say that it is $P$-piecewise scalable.  
 Thus, a frame $\F$ is piecewise scalable if there exist an orthogonal projection $P$ and constants $\{a_i, b_i\}_{i=1}^m $ such that for every $x\in\R^n$, 
 $$
 \sum_{i=1}^m |\langle x, a_iPx_i+b_i(I-P)x_i\rangle|^2 = \|x\|^2.
 $$
  Equivalently, a frame is piecewise scalable if there exist   an  orthogonal projection $P$ and  diagonal matrices  $D_1 $ and $D_2 $ for which the columns of  $PXD_1 + (I-P)XD_2 $ form a Parseval frame.
  Clearly, if $D_1=D_2$ we are back to the definition of scalable frames. So scalable frames are piecewise scalable.
  Note that  the piecewise scalability of a frame $\F=\{x_i\}_{i=1}^m$ is preserved under the replacement $x_i$ with $-x_i$ for any $i$. 

\medskip
It is not difficult to find  piecewise scalable frames which are not scalable. For example, in $\mathbb{R}^2$, a frame $\{x_i\}_{i=1}^2$ is scalable if and only if the
vectors are orthogonal. But we have a strong result for piecewise scalable.

\begin{theorem}
A frame $\{x_i\}_{i=1}^2$ in $\mathbb{R}^2$ is piecewise scalable by every non-trivial orthogonal projection $P$.
\end{theorem}

\begin{proof}
Let $P$ be a non-trivial orthogonal projection on $\mathbb{R}^2$. It is not possible that $Px_1=Px_2=0$ or $P$ is trivial. Also, it is not possible that 
$Px_i=0=(I-P)x_i$ or $x_i=0$. It follows that there is $i\not= j \in \{1,2\}$ so that $Px_i\not= 0\not= (I-P)x_j$. So letting $a_i=\frac{1}{\|Px_i\|},\ a_j=0$ and
$b_i=0,\ b_j=\frac{1}{\|(I-P)x_j\|}$, then 
$\{a_iPx_i+b_i(I-P)x_i\}_{i=1}^2$ is an orthonormal basis for $\RR^2$, which is a Parseval frame.
\end{proof}

Another reason why piecewise scalable frames are interesting is that they are  both  a generalization and a special case  of  scalable frames. Indeed,   if frame $\F$ is  scalable, then it is also piecewise scalable (with constants $a_i=b_i$) with respect to any orthogonal projections.   Also, if  $\F$ is  piecewise scalable  with an orthogonal projection $P$, 
then the sets  $ \{Px_i\}_{i=1}^m$ and  $\{ (I-P)x_i\}_{i=1}^m$  are scalable frames  in the Hilbert spaces  $P(\R^n)$ and $(I-P)(\R^n)   $, respectively, (see Theorem \ref{T-piecewise-1}).

Throughout, we will denote vectors  $x\in\R^n$ with  $x= (x(1), \ldots, x(n))$.  We denote with $e_1=(1, 0, \ldots, 0), \ldots,  e_n= (0, \ldots, 0, 1)$ the vectors of the canonical basis of $\R^n$. When  $m\in\NN$, we use $[m]$ to denote the set $\{1,\ldots, m\}$.

The remainder of the paper is organized as follows. In Section 2 we give several necessary and sufficient conditions for a frame to be piecewise scalable. In Section 3, we show that piecewise scalability is preserved under unitary transformations and give some related results. In Section 4 we show that frames in $\R^3$ are always piecewise scalable.  We also show that if the frame vectors are close to each other, then it cannot be piecewise scalable for almost all orthogonal projections. Finally, in Section 5, we present some properties of scaling constants.

\section{Some results and examples }

In this section we prove necessary and sufficient conditions for piecewise scalabilty of frames in $\RR^n$. The main result of this section is the following.
\begin{theorem}\label{T-piecewise-1} 
	Let $\F=\{x_i\}_{i=1}^m$ be a frame for $\RR^n$. The following are equivalent:
	\begin{enumerate}
		\item $\F$ is piecewise scalable.
		\item There is an orthogonal projection $P$ and scalars $\{a_i\}_{i=1}^m$ and $\{b_i\}_{i=1}^m$ so that:
		\begin{enumerate}
			\item Both $\{Px_i\}_{i=1}^m$ and $\{(I-P)x_i\}_{i=1}^m$ are scalable frames with scaling constants $\{a_i\}_{i=1}^m$ and
			$\{b_i\}_{i=1}^m$ in the range of $P$ and $I-P$, respectively.
			\item For all $x\in \RR^n$, we have that 
			\[\sum_{i=1}^m a_ib_i\langle x,Px_i\rangle \langle x,(I-P)x_i\rangle = 0. \]
		\end{enumerate}
	\end{enumerate}
\end{theorem}

\begin{proof} 
	$(1)\Rightarrow(2)$:
	By (1), there are scalars $\{a_i\}_{i=1}^m$ and $\{b_i\}_{i=1}^m$ and an orthogonal projection $P$ so that $\{a_iPx_i+b_i(I-P)x_i\}_{i=1}^m$ is a Parseval
	frame. Since the projection of a Parseval frame is a Parseval frame, and
	$$ P\left ( a_iPx_i+b_i(I-P)x_i \right )= a_iPx_i,$$
	it follows that $\{a_iPx_i\}_{i=1}^m$ is a Parseval frame. Similarly, $\{b_i(I-P)x_i\}_{i=1}^m$ is a Parseval frame and (2)(a) is proved.

We now prove (2)(b). Let $x\in\R^n$;  in view of (2)(a) and the fact that $\{a_iPx_i+b_i(I-P)x_i\}_{i=1}^m$ is a Parseval frame, we can write the following chain of identities. 
	\begin{align*}
		\|x\|^2&=\sum_{i=1}^m|\langle x,a_iPx_i+b_i(I-P)x_i \rangle|^2\\
		&= \sum_{i=1}^m| \langle x,a_iPx_i\rangle +\langle x,b_i(I-P)x_i\rangle|^2\\
		&= \sum_{i=1}^m|\langle x,a_iPx_i\rangle |^2+\sum_{i=1}^m|\langle x,b_i(I-P)x_i\rangle|^2+
		2\sum_{i=1}^m\langle x,a_iPx_i\rangle\langle x,b_i(I-P)x_i\rangle\\
		&= \|Px\|^2+\|(I-P)x\|^2 + 2\sum_{i=1}^m\langle x,a_iPx_i\rangle\langle x,b_i(I-P)x_i\rangle\\
		&= \|x\|^2+  2\sum_{i=1}^m\langle x,a_iPx_i\rangle\langle x,b_i(I-P)x_i\rangle.
	\end{align*}
	Thus, $ \sum_{i=1}^m\langle x,a_iPx_i\rangle\langle x,b_i(I-P)x_i\rangle=0$  and  (2)(b) is proved.
	
	\vskip12pt
	$(2)\Rightarrow (1)$: We need to show that $\{a_iPx_i+b_i(I-P)x_i\}_{i=1}^m$ is a Parseval frame. Let $x\in\RR^n$; using (2)(a) and (2)(b), we obtain
	\begin{align*}
		&\quad \sum_{i=1}^m|\langle x,a_iPx_i +b_i(I-P)x_i \rangle|^2\\
		&= \sum_{i=1}^m|\langle x,a_iPx_i\rangle |^2+\sum_{i=1}^m|\langle x,b_i(I-P)x_i\rangle|^2+ 
		2\sum_{i=1}^m\langle x,a_iPx_i\rangle\langle x,b_i(I-P)x_i\rangle\\
		&= \|Px\|^2+\|(I-P)x\|^2 +  
		2\sum_{i=1}^ma_ib_i\langle x, Px_i\rangle\langle x, (I-P)x_i\rangle 
		 = \|x\|^2.
	\end{align*}
	So (1) is proved.
\end{proof}

\begin{corollary} \label{Cor1}
 
	 If $\F=\{x_i\}_{i=1}^m\subset \R^n$ is   piecewise scalable  with projection $P$, then the set  $\{Px_i\}_{i=1}^m\cup \{ (I-P)x_i\}_{i=1}^m$ is a scalable frame for $\R^n$.
	 	
\end{corollary}
\begin{proof}
By Theorem  \ref{T-piecewise-1} (2)(a), there exist constants $\{a_i, b_i\}_{i=1}^m$ so that both  
$\{a_iPx_i\}_{i=1}^m$ and $\{b_i(I-P)x_i\}_{i=1}^m$ are Parseval frames for  $P(\R^n)$ and \newline $(I-P)(\R^n)$, respectively.  Thus, for any $x\in \R^n$ we have that
\begin{align*}
	&\sum_{i=1}^m |\langle  x, a_iPx_i\rangle  |^2+ \sum_{i=1}^m |\langle x, b_i(I-P)x_i\rangle|^2
	\\
	= &\sum_{i=1}^m |\langle Px, a_iPx_i\rangle   |^2+ \sum_{i=1}^m |\langle(I-P)x, b_i(I-P)x_i\rangle|^2
\\   = & \|Px\|^2+ \|(I-P)x\|^2= \|x\|^2.   
\end{align*}
as required
 \end{proof}

  \begin{corollary}\label{Cor2}   Let $\F=\{x_i\}_{i=1}^m$ be a frame for $\R^n$. If there  exist an orthogonal projection $P$ and a subset $I\subset [m]$ such that $\{Px_i\}_{i\in I}$ and  $\{(I-P)x_i\}_{i\in I^c}$ are scalable in the range of $P$ and $ I-P$  respectively, then $\F$ is piecewise scalable.

   \end{corollary}
\begin{proof}
 Let $\{a_i\}_{i\in I}$ and $\{b_i\}_{i\in I^c}$ so that $\{a_iPx_i\}_{i\in I}$ and $\{b_i(I-P)x_i\}_{i\in I^c}$ are Parseval frames for $P(\RR^n)$ and $(I-P)(\RR^n)$, respectively.  
  Let $a_i=0$ if  $i\in I^c$ and $b_i= 0$ if $i\in I$. Then, the sets $\{ Px_i\}_{i=1}^m$ and $\{ (I-P)x_i\}_{i=1}^m$ are scalable frames  in the range of $P$ and $ I-P$  with constants $\{a_i\}_{i=1}^m$ and $\{b_i\}_{i=1}^m$, and $a_ib_i=0$ for all $i\in [m]$. The conclusion follows from Theorem \ref{T-piecewise-1}.
\end{proof}

The following is a special case of Corollary \ref{Cor2} which we will use later on.

\begin{corollary}\label{C3}
	Let $\F=\{x_i\}_{i=1}^m$ be a frame for $\RR^n$.  If there exist a set $I\subset [m]$ of size $n-1$ and an orthogonal projection $P$ on $\R^n$  such that  
$\{Px_i\}_{i\in I}$ is  an orthogonal set of nonzero vectors, and $(I-P)x_j$ is nonzero for some $j\notin I$, then $\F$ is piecewise scalable.  
\end{corollary}
 
 \begin{ex}  \label{exam2.5}

There exist bases for $\RR^n$ which are non-scalable but are piecewise scalable.

Indeed,  let $\{y_i\}_{i=1}^{n-1}$ be an orthogonal basis for $\RR^{n-1}$. 
Let $\epsilon >0$ and let
$$ x_i=(y_i, \epsilon) \mbox{ for } i=1, \ldots, n-1, \mbox{ and } x_{n}= (0,\, ...,\,0, 1).
$$ 
These vectors are linearly independent since $\det [x_i]_{i=1}^{n}=\det[y_i]_{i=1}^{n-1}\not=0$.

The set  $\{x_i\}_{i=1}^n$ is not scalable because the $x_i$ are not orthogonal,
but  it is a piecewise scalable frame for $\RR^n$. To see this, let $P$ be the orthogonal projection on the first $n-1$ coordinates. Then $\{Px_i\}_{i=1}^{n-1}=\{y_i\}_{i=1}^{n-1}$, which is an orthonormal set, and $(I-P)x_n=x_n$. By Corollary \ref{C3}, $\{x_i\}_{i=1}^n$ is piecewise scalable.
\end{ex}

With the same idea as in Example \ref{exam2.5}, we can construct many other examples of piecewise scalable frames with more vectors than dimension.

\begin{ex}
Let $\{y_i\}_{i=1}^m$ be a scalable frame for $\R^{n-1}$ and let $\{e_i\}_{i=1}^n$ be the canonical basis of $\RR^n$. Let $\{x_i\}_{i=1}^{m+1}$ be the frame for $\RR^n$ defined by

\[x_i=(y_i, \epsilon_i), i \in [m] \mbox { and } x_{m+1}=e_n.\] Then $\{x_i\}_{i=1}^{m+1}$ is a piecewise scalable frame for $\RR^n$.
\end{ex}

We mentioned in the introduction that one problem with scaling is that most of the time, scaling is done by setting most of the scaling coefficients
equal to zero. This problem does not always occur in piecewise scaling, as the following example shows. In this example in $\RR^4$, we have an arbitrarily
large number of vectors which have all coefficients but four equal to zero for scaling while all coefficients are non-zero for piecewise scaling.

\begin{ex}
Let $\{x_n\}_{n=1}^4$ be an orthonormal basis for $\RR^4$ and let $\{y_{2i},y_{2i+1}\}_{i=1}^m$ be a family of orthonormal bases for $\RR^2$.
For $n=1,2,\ldots,m$, we define a set of vectors in $\RR^4$ as in the following:
\[ x_{4n+1}=(y_{2n}, 2y_{2n}), \ x_{4n+2}=(y_{2n},-2y_{2n}),\]
\[ x_{4n+3}=(y_{2n+1},2y_{2n+1}),\ x_{4n+4}=(y_{2n+1},-2y_{2n+1}).\]
Consider the frame for $\RR^4$ given by: $ \{x_n\}_{n=1}^{4(m+1)}$. 
This frame is easily scalable with $a_i=1$ for $i=1,2,3,4$ and $a_i=0$ for $i\ge 5$.

Now we look at piecewise scaling. If we let $P$ be the orthogonal projection onto $\spn\{e_1,e_2\}$, then $\{Px_n\}_{n=5}^{4(m+1)}$ is a set of orthonormal bases for $\RR^2$. And 
$\{\frac{1}{2}(I-P)x_n\}_{n=5}^{4(m+1)}$ is also a set of orthonormal bases for $\spn\{e_3,e_4\}$. So we let
$a_i=b_i=1$ for $i=1,2,3,4$ and $a_i=\frac{1}{\sqrt{2}}$ and $b_i=\frac{1}{2\sqrt{2}}$ for  for $i\ge 5$. Then $\{a_nPx_n+b_n(I-P)x_n\}_{n=1}^{4(m+1)}=
\{z_n\}_{n=1}^{4(m+1)}$ where $\{z_i\}_{i=1}^4=\{x_i\}_{i=1}^4$ is an orthonormal basis for $\RR^4$ and
\[ z_{4n+1}=\dfrac{1}{\sqrt{2}}(y_{2n}, y_{2n}),\ z_{4n+1}=\dfrac{1}{\sqrt{2}}(y_{2n},- y_{2n}),\]
\[\ z_{4n+3}=\dfrac{1}{\sqrt{2}}(y_{2n+1}, y_{2n+1}),\ z_{4n+4}=\dfrac{1}{\sqrt{2}}(y_{2n+1},- y_{2n+1}).\]
 Note that for each $n\in [m]$, these four vectors form an
orthonormal basis for $\RR^4$.  So $\{a_nPx_n+b_n(I-P)x_n\}_{n=1}^{4(m+1)}$ is a set of $(m+1)$-orthonormal bases for $\RR^4$ and so multiplying all the vectors
by $\frac{1}{\sqrt{m+1}}$ produces a Parseval frame for $\RR^4$.
\end{ex}

\medskip

	Let $\F=\{x_i\}_{i=1}^m$ be a frame for $\RR^n$. Let $P$ be an orthogonal projection of $\RR^n$ and let $Y =P(\RR^n)$ and $Z =(I-P)(\RR^n)$.   Let $T_1,T_2$ be
	the analysis operators of $\{Px_i\}_{i=1}^m$ and $\{(I-P)x_i\}_{i=1}^m$, respectively.
		 We extend these operators to $\RR^n$ by defining them to be zero on the orthogonal complement of their domains. Also, if $V$ is a subspace of $\RR^n$, $I_V$ is the operator which is the identity on $V$ and zero on $V^{\perp}$.
	
	 The following theorem can be viewed as an  ``operator form" of Theorem \ref{T-piecewise-1}.
	
	\begin{theorem}\label{T2} With the notation and the definitions stated above,   the following are equivalent:
	\begin{enumerate}
		\item $\F$ is piecewise scalable with the projection $P$ and scaling  constants $\{a_i, b_i\}_{i=1}^m$.
		\item We have
		\[ T_1^*D_1^2T_1=I_{Y},\quad  T_2^*D_2^2T_2=I_{Z}  \mbox{ and }T_1^*D_1D_2T_2=T_2^*D_2D_1T_1=0,
		\] where $D_1,D_2$ are diagonal operators on $\ell_2(m)$ with diagonal elements $\{a_i\}_{i=1}^m$ and $\{b_i\}_{i=1}^m$ respectively.
	\end{enumerate}

\end{theorem}

\begin{proof}
	It is easy to verify that the 
	analysis operator of the frame $\{a_iPx_i+b_i(I-P)x_i \}_{i=1}^m$ is $T=D_1T_1+D_2T_2$. So the synthesis operator is
	$T^*=T_1^*D_1+T_2^*D_2$ and the frame operator is $S= (T_1^*D_1+T_2^*D_2)( D_1T_1+D_2T_2)$. 
	
	$(1)\Rightarrow (2)$:
	Since $\{a_iPx_i+b_i(I-P)x_i\}_{i=1}^m$ is Parseval, we have that $S=I$ and $\{a_iPx_i\}_{i=1}^m$ and $\{b_i(I-P)x_i\}_{i=1}^m$ are Parseval frames for $Y$ and $Z$, respectively. Therefore,
	\begin{align*}
		I&=(T_1^*D_1+T_2^*D_2)( D_1T_1+D_2T_2)\\
		&= T_1^*D_1^2T_1+T_2^*D_2^2T_2+T_1^*D_1D_2T_2+T_2^*D_2D_1T_1\\
		&= I_{Y}+ I_{Z}+T_1^*D_1D_2T_2+T_2^*D_2D_1T_1\\
		&=I+T_1^*D_1D_2T_2+T_2^*D_2D_1T_1.
	\end{align*}
	So $T_1^*D_1D_2T_2+T_2^*D_2D_1T_1=0$. Since $T_1=0$  on $Z$ and $T_2=0$ on $Y$, we have that  $T_1^*D_1D_2T_2+T_2^*D_2D_1T_1 $=0 if and only if $T_1^*D_1D_2T_2=T_2^*D_2D_1T_1 $=0.
	\vskip12pt
	$(2)\Rightarrow (1)$: Given (2) it is immediate that $S=I$.	 
\end{proof}

\begin{corollary}  A frame  $\F=\{x_i\}_{i=1}^m\subset \R^n$ is  piecewise   scalable with an orthogonal  projection $P$ and scaling constants $\{a_i, b_i\}_{i=1}^m$  if and only if $\{a_iPx_i\}_{i=1}^m$ and $\{b_i(I-P)x_i\}_{i=1}^m$ are Parseval frames in  the range of $P$ and $I-P$, respectively, and 
 the rows of the matrix   $ [a_iPx_i]_{i=1}^m $ are orthogonal  to the rows of  $[b_i(I-P)x_i ]_{i=1}^m$ .

\end{corollary}
 
\begin{proof}	
	 By Theorem \ref{T-piecewise-1},  it is enough to show that  the conditions (1) and (2)  below are equivalent:
	 	
\begin{enumerate} \item The rows of the matrix   $ [a_iPx_i]_{i=1}^m $ are orthogonal  to the rows of  $[b_i(I-P)x_i ]_{i=1}^m$. 
	 	
	 	\medskip
	\item $\sum_{i=1}^m a_ib_i\langle x,Px_i\rangle \langle x,(I-P)x_i\rangle = 0. $
	 	
	 	\end{enumerate} 
	
	Let $T_1, T_2, D_1, D_2$ be operators defined as in Theorem \ref{T2}. Then (1) is equivalent to $T_1^*D_1D_2T_2=T_2^*D_2D_1T_1=0$. We now show that the latter is equivalent to (2). 
	
	Indeed, for any $x\in \RR^n$, we have that \[\sum_{i=1}^m a_ib_i\langle x,Px_i\rangle \langle x,(I-P)x_i\rangle = \langle T_1^*D_1D_2T_2x, x\rangle =\langle T^*_2D_2D_1T_1x, x\rangle. \]
	Hence, if $T_1^*D_1D_2T_2=T_2^*D_2D_1T_1=0$ then obviously, we have (1).

	 Conversely, suppose we have (1); then for all $x\in \RR^n$, 
	\[ \langle T_1^*D_1D_2T_2x, x\rangle =\langle T^*_2D_2D_1T_1x, x\rangle=0. \] This implies 
	\[\langle (T_1^*D_1D_2T_2+T_2^*D_2D_1T_1)x, x\rangle =0,\] for all $x\in \RR^n$. Since $T_1^*D_1D_2T_2+T_2^*D_2D_1T_1$ is symmetric, it follows that it is zero,  
		and so $T_1^*D_1D_2T_2=T_2^*D_2D_1T_1=0$. The proof is complete.
\end{proof}


\section{Standard piecewise scalable frames}
Let $\{e_i\}_{i=1}^n$ be the canonical orthonormal basis for $\R^n$. For any subset $J\subset [n]$, we call the orthogonal projection onto $\spn\{e_i\}_{i\in J}$, which is denoted by $\Pi_J$, the canonical projection of $\R^n$. For example, for any $1\leq k\leq n$, 
$ \Pi_{[k]}x= (x(1),\ldots, x(k), 0,\ldots,0) $ is the canonical projection onto $\spn\{ e_1,\ldots, e_k\}$. 

A special case of Definition \ref{D1} is the following.

\begin{definition}\label{D2}  Given a  frame $\F=\{x_i\}_{i=1}^m$ for $\RR^n$, 
we say that $\F$ is a standard piecewise scalable frame if there exists a canonical projection $\Pi_J$ of $\R^n$ such that $\F$ is piecewise scalable with 
 $\Pi_J$.
\end{definition}

   Let $\Pi_J$ be a canonical projection of rank $k\ge 1$.  After perhaps  a change of coordinates, we can assume that $J=[k]$. For $x_i\in \mathbb{R}^n$ we let 
  \[y_i=\Pi_{[k]}x_i=(x_i(1), x_i(2), \ldots, x_i(k), 0, \ldots, 0)\] and \[z_i=(I-\Pi_{[k]})x=(0, \ldots, 0, x_i(k+1), \ldots, x_i(n)).\]
  
  Thus, by our definition, a frame $\F=\{x_i\}_{i=1}^m$ for $\RR^n$ is piecewise scalable with the projection $\Pi_{[k]}$ if there exist scalars $\{a_i, b_i\}_{i=1}^m$ so that $\{a_iy_i+b_iz_i\}_{i=1}^m$ is a Parseval frame for $\RR^n$. We refomulate Theorem \ref{T-piecewise-1} for this case.
  
  \begin{theorem}\label{thm canonical proj} 
  	Let $\F=\{x_i\}_{i=1}^m$ be a frame for $\RR^n$. For a fixed $0<k<n$, let $y_i=\Pi_{[k]}(x_i)$ and $z_i=(I-\Pi_{[k]})x_i, i\in [m]$. The following are equivalent:
  	\begin{enumerate}
  		\item  There exist scalars $\{a_i, b_i\}_{i=1}^m$ for which 
  		$\{a_iy_i+b_iz_i\}_{i=1}^m$ is a Parseval frame for $\RR^n$.
  		\item There are scalars $\{a_i\}_{i=1}^m$ and $\{b_i\}_{i=1}^m$ satisfying the following two conditions:
  		\begin{enumerate}
  			\item $\{a_iy_i\}_{i=1}^m$ and $\{b_iz_i\}_{i=1}^m$ are both Parseval frames for their span.
  			\item For every $1\le j \le k$ and every $k+1\le \ell \le m$ we have
  			$\displaystyle \sum_{i=1}^m a_iy_i(j)b_iz_i(\ell)=0.$
  		Equivalently,  the $j^{th}$-row of the matrix $[a_iy_i+b_iz_i]_{i=1}^m$   is orthogonal to the $\ell^{th}$-row for all
  			$1\le j \le k $ and $ k+1\le \ell \le n$.
  		\end{enumerate}
  		\item There are scalars $\{a_i\}_{i=1}^m$ and $\{b_i\}_{i=1}^m$ satisfying the following two conditions:
  		\begin{enumerate}
  			\item $\{a_iy_i\}_{i=1}^m$ and $\{b_iz_i\}_{i=1}^m$ are both Parseval frames for their span.
  			\item We have
  		$\displaystyle \sum_{i=1}^ma_ib_i\langle x,y_i\rangle \langle x,z_i\rangle =0\mbox{ for all }x\in\RR^n.$
  		\end{enumerate}
  	\end{enumerate}
  	
  \end{theorem}

\begin{remark} Condition (2)(b) in Theorem \ref{thm canonical proj} is necessary. The following example shows this.
\end{remark}

\begin{ex}\label{E3}
	Let $\F=\{x_i\}_{i=1}^4$ be a frame for $\RR^4$, where
	
	\[ x_1=(1,0,1,1), \ x_2= (0,1,1,-1),\ x_3=e_3 ,\ x_4=e_1 .\]
	Let $k=2$. Then
	\[ y_1=(1,0,0,0),\ y_2=(0,1,0,0),\ y_3=(0,0,0,0),\ y_4=(1,0,0,0),\]
	and
	\[ z_1=(0,0,1,1),\ z_2=(0,0,1,-1), \ z_3=(0,0,1,0),\ z_4=(0,0,0,0).\] So if we choose $a_1=a_2=1$, $a_3=a_4=0$, and $b_1=b_2=\frac{1}{\sqrt{2}}$, $b_3=b_4=0$ then $\{a_ix_i\}_{i=1}^4=
	\{y_1, y_2\}$ is a Parseval frame for $ \spn\{e_1,e_2\}$ and $\{b_iz_i\}_{i=1}^4=\{{\frac{1}{\sqrt{2}}}z_1,{\frac{1}{\sqrt{2}}}z_2\}$ is a Parseval frame for
	$ \spn\{e_3,e_4\}$. But we can see that the set of vectors, $\{a_iy_i+b_iz_i\}_{i=1}^4$, i.e., the set of the   columns of the matrix
	\[A=\begin{pmatrix}
		a_1&0&0&a_4\\
		0&a_2&0&0\\
		b_1&b_2&b_3&0\\
		b_1&-b_2&0&0
	\end{pmatrix}\] is not a Parseval frame for $\RR^4$ for any choices of $\{a_i, b_i\}_{i=1}^4$. Note that if the $a_i$, $b_i$ are chosen as above,    the columns of the matrix $A$  do not even form a frame.
\end{ex}

We note that a frame can be piecewise scalable with one but not all canonical orthogonal projections. Also, a frame may not be piecewise scalable with any canonical orthogonal projections.

In the following example, we will use the fact \cite{KO} that in $\mathbb{R}^2$, a spanning set of vectors which lie in an open quadrant is not scalable.
\begin{ex}\label{E4}
	Let $\F=\{x_i\}_{i=1}^m$ be a frame for $\RR^3$, satisfying
\[ x_i(j)>0,\mbox{ for all }i\in [m] \mbox{ and } j\in [n].\]	
Then any canonical projection $\Pi_J$, $\{\Pi_Jx_i\}_{i=1}^m$ and $\{(I-\Pi_J)x_i\}_{i=1}^m$ both lie in a quadrant so at least one is not scalable.

\end{ex}

If a frame $\F$ is scalable, then every frame obtained from $\F$ through a unitary transformation   is scalable as well \cite{KO}. 
The next theorem  shows that also the   property of  being piecewise scalable is  preserved  by unitary transformations.  We prove first the following.
\begin{lemma}\label{lem1}
	Let $\F=\{x_i\}_{i=1}^m$ be a  piecewise scalable frame for $\RR^n$ with an orthogonal projection $P$. Let  $Q:\RR^n\to
	\RR^n$ be an orthogonal projection and let $U$ be a unitary transformation such that $UP= QU$. Then $U(\F)$ is  piecewise scalable with projection $Q$  and with the same scaling constants.
\end{lemma}
\begin{proof} By assumption,  there exist constants $\{a_i, b_i\}_{i=1}^m$ so that $\{a_iPx_i+b_i(I-P)x_i\}_{i=1}^m$ is a Parseval frame. We show that 
	 $\{a_iQUx_i+b_i(I-Q)Ux_i\}_{i=1}^m$ is a Parseval frame as well. For every $x\in \RR^n$,  
	\begin{align*}
		\sum_{i=1}^m|\langle x,\ a_iQUx_i  +b_i(I-Q)Ux_i)\rangle|^2  &=\sum_{i=1}^m|\langle x,\ a_iUPx_i + b_iU(I-P)x_i\rangle|^2\\
		&=\sum_{i=1}^m|\langle U^*x,\ a_iPx_i + b_i(I-P)x_i\rangle|^2\\
		&=\|U^*x\|^2 
		 =\|x\|^2.
	\end{align*}
	Thus, $\{a_iQUx_i+b_i(I-Q)Ux_i\}_{i=1}^m$ is a Parseval frame, as required.
\end{proof}

\begin{theorem}\label{T1} Let  $\F=\{x_i\}_{i=1}^m$ be a piecewise scalable frame for $\R^n$  with an orthogonal projection $P$ of  rank $k$. Then the following hold:
	\begin{enumerate}
		\item For every unitary transformation $U:\RR^n\to\RR^n$, the frame  $U(\F)$ is   piecewise scalable. 
                  \item For every orthogonal projection $Q$ of rank $k$, there exists a unitary transformation $U$ such that $U(\F)$ is $Q$-piecewise scalable. 
	\end{enumerate}
\end{theorem}
\begin{proof}
	(1): Let $Q= UPU^*$. Since $QU=UP$,   Lemma \ref{lem1}  yields that $U(\F)$ is piecewise scalable with projection $Q$.

(2): Let $\{u_i\}_{i=1}^k$ and $\{v_i\}_{i=1}^k$ be orthonormal bases of $P(\R^n)$ and $Q(\R^n)$, respectively. We complete these sets to   orthonormal bases of $\R^n$, that we denote with  $\{u_i\}_{i=1}^n$ and $\{v_i\}_{i=1}^n$.  We define a linear map $U :\RR^n\to\RR^n$  by $Uu_i = v_i, i\in [n]$ 
	and we extend it by linearity to the whole $\RR^n$.
The operator $U$ is unitary because it maps an orthonormal basis to an orthonormal basis.  
We  show that $UP=QU$ and hence (2) follows by Lemma \ref{lem1}. 

It is enough to show that $(UP)u_i=(QU)u_i$ for every $i\in [n]$. 
Indeed, when $i\leq k$, $(UP)u_i= Uu_i=v_i=Qv_i=(QU)u_i$  and $(UP)u_i=0=Qv_i=(QU)u_i$ when $i>k$. The proof is complete.
	\end{proof}

 \begin{corollary}\label{cor1} A frame $\F$ for $\RR^n$ is piecewise scalable if and only if there exists a unitary transformation $U:\R^n\to\R^n$ such that $U(\F)$ is $\Pi_{[k]}$-piecewise scalable for some $k\leq n$.
 \end{corollary}
 \begin{proof} If $\F$ is piecewise scalable with a projection $P$  of rank $k$, by Part 2 of Theorem \ref{T1}  we can choose a unitary transformation $U$ so that  $U(\F)$ is piecewise scalable with projection $\Pi_{[k]}$.

 	Now suppose that there exists a unitary transformation $U$ such that $U(\F)$ is $\Pi_{[k]}$-piecewise scalable. By Part 1 of Theorem \ref{T1}, $U^{-1}U(\mathcal{X})=\mathcal{X}$ is piecewise scalable. 
 \end{proof}

  \section{Piecewise scalability of frames in low dimensional spaces}

In this section we prove that all frames in  $\R^3$ are piecewise scalable. In $\RR^4$, we show that   frames  whose vectors are close to each other (in  the sense of Corollary \ref{R4}) are not piecewise scalable with any projection of rank 2. We then generalize this result in higher dimensional spaces.

 \begin{theorem}\label{n=3}
 	Every frame for $\R^3$ is piecewise scalable.
 \end{theorem}

\begin{proof} It is enough to show that any frame with three vectors in $\R^3$ is piecewise scalable. Let   $\F=
 \{x_1, x_2, x_3\}$ be a unit-norm frame for $\RR^3$.

Without loss of generality we can assume that $a:=\langle x_1, x_2\rangle \geq0$. Note that $a<1$. Let $u=\lambda (x_1+x_2) +z, \mbox{ where } z\perp \spn\{x_1, x_2\}, \|z\|=1$, and $\lambda = \pm\sqrt{\frac{a}{1-a^2}}$. Let $P$ be the orthogonal projection onto $\spn\{u\}$. So
\[Px=\frac{1}{\|u\|^2}\langle x, u\rangle u, \mbox { for } x\in \RR^3.\]
We first show that $\langle (I-P)x_1, (I-P)x_2\rangle =0$. To see this, we compute:
\[\|u\|^2=2\lambda^2(1+a)+1\] and so
\begin{align*}
\langle (I-P)x_1, (I-P)x_2\rangle&=\langle x_1, x_2\rangle -\langle Px_1, Px_2\rangle\\
&=a - \frac{1}{\|u\|^2}\langle x_1, u\rangle \langle x_2, u\rangle\\
&=a-\frac{1}{\|u\|^2}\langle x_1,\lambda(x_1+x_2)\rangle \langle x_2,\lambda(x_1+x_2)\rangle \\
&=a-\frac{1}{\|u\|^2}\lambda^2(1+\langle x_1, x_2\rangle)^2\\
&=a-\dfrac{\frac{a}{1-a^2}(1+a)^2}{2\frac{a}{1-a^2}(1+a)+1}\\
&=a-\dfrac{a(1+a)}{1+a}=0.
\end{align*}

Now we check that $(I-P)x_1, (I-P)x_2$ are nonzero and $Px_3\not=0$. 

Indeed, if $(I-P)x_1= 0$ then 
\[x_1=\frac{\langle x_1, u\rangle}{\|u\|^2}u = \frac{\langle x_1, u\rangle}{\|u\|^2}[\lambda(x_1+x_2)+z].\] This implies that $z\in \spn\{x_1, x_2\}$, which is imposible by the choice of $z$. Similarly, we have that $(I-P)x_2\not=0$. 

To see $Px_3\not=0$, note that we have two choices of $\lambda$ to get the vector $u$. One of the choices must satisfy $Px_3\not=0$; otherwise, $x_3$ is orthogonal to  $\pm \sqrt{\frac{a}{1-a^2}}(x_1+x_2)+z$ and so $x_3\perp z$, a contradiction since $\F$ spans $\RR^3$. Thus, $\{(I-P)x_1, (I-P)x_2, Px_3\}$ is an orthogonal set of nonzero vectors. This implies that $\F$ is piecewise scalable with the projection $P$.
\end{proof}

 \medskip

We have proved   that every frame in $\R^3$ is piecewise scalable by choosing a linearly independent subset of the frame, say  $\{x_1, x_2, x_3\}$, and then constructing   an orthogonal projection $P$ so that $Px_1, Px_2$ are nonzero and orthogonal, and $(I-P)x_3\not=0$.

The same strategy would work  in  $\RR^4$ if, for any given linearly independent set $\{x_1, x_2, x_3, x_4\}$ in $\R^4$, one of the following  holds.
\begin{enumerate}
\item There is an orthogonal projection $P$ of rank 2 so that $Px_1, Px_2$ are nonzero and orthogonal, and $(I-P)x_3, (I-P)x_4$ are nonzero and orthogonal.
\item There is an orthogonal projection $P$ of rank 3 so that $Px_1, Px_2$ and $Px_3$ are nonzero and orthogonal, and $(I-P)x_4\not=0$.
\end{enumerate}
Let us analyze the first case. 
To construct the projection $P$, we need to find   orthonormal vectors $u$, $v \in\R^4$   such that   the orthogonal projection  onto $\spn\{u, v\}$, i.e.,  
\[Px=\langle x, u\rangle u+\langle x, v\rangle v , \qquad  x\in \RR^4, \]
 satisfies  $Px_1$, $Px_2\ne 0$ and $(I-P)x_3$, $ (I-P)x_4\ne 0$, and also 
$$Px_1 \perp Px_2 \mbox{ and }  (I-P)x_3 \perp (I-P)x_4.$$
The last  2 conditions are equivalent to 
\[\langle x_1, u\rangle\langle x_2, u\rangle+\langle x_1, v\rangle\langle x_2, v\rangle=0,\]
 and 
\[\langle x_3, u\rangle\langle x_4, u\rangle+\langle x_3, v\rangle\langle x_4, v\rangle=\langle x_3, x_4\rangle.\]

Thus, in order to find a desirable $P$, we need to  find vectors $u$, $v$ by solving a system of 5 equations with 8 variables. 
In principle such  a system   has  a great chance of having solutions, but unfortunately, this is not always the case.

\begin{lemma} \label{lem proj}
Let  $x,\, y\in \R^n$ with $\|x\|=\|y\|=1$ and $\epsilon>0$.  If $\|x-y\|<\epsilon$, then for any  orthogonal projection  $P:\RR^n\to \R^n$, we have that
\[ \max\{\langle Px,Py\rangle,\langle (I-P)x,(I-P)y\rangle\}\ge \frac{1}{2}-\epsilon.\]
\end{lemma}

\begin{proof}
Compute:
\[ \langle Px,Py\rangle=\langle Px,Px\rangle+\langle Px,P(y-x)\rangle\ge \|Px\|^2-\epsilon.
\]
Similarly,
\[ \langle (I-P)x,(I-P)y\rangle \ge \|(I-P)x\|^2-\epsilon.\]
But
\[ 1=\|Px\|^2+\|(I-P)x\|^2\mbox{ so } \max\{\|Px\|^2,\|(I-P)x\|^2\}\ge \frac{1}{2}.\]
These three inequalities prove the lemma.
\end{proof}

\begin{theorem}\label{thm proj}
Let $0<  \epsilon<1 $ and let $\{x_i\}_{i=1}^m$ be a unit-norm frame for $\RR^n$ such that $\|x_i-x_j\|< \epsilon$   for all $i\not= j$. Let $P$ be an
orthogonal projection on $\RR^n$. One of the following must hold:
\begin{enumerate}
\item $\langle Px_i,Px_j\rangle \ge \dfrac{1}{2}-4\epsilon, \mbox{ for all }i\not= j.$
\item $\langle (I-P)x_i,(I-P)x_j\rangle \ge \dfrac{1}{2}-\epsilon, \mbox{ for all }i\not= j.$
\end{enumerate}
\end{theorem}

\begin{proof}
We   assume that (2) fails and show that (1) holds. After re-indexing, we may assume
\[ \langle (I-P)x_1,(I-P)x_2\rangle < \frac{1}{2}-\epsilon.\]
By Lemma \ref{lem proj},
\[ \langle Px_1,Px_2\rangle \ge \frac{1}{2}-\epsilon.\]
Fix $1\le i\not= j \le m$ and let $x=x_i-x_1,\ y=x_j-x_2$. By assumption, $$|\langle Px_1,Py\rangle| \le \epsilon, \quad  |\langle Px_2,Px\rangle| \le \epsilon, \ \mbox{ and } 
  |\langle Px,Py\rangle |\le \epsilon^2.$$
Therefore
\begin{align*}
\langle Px_i,Px_j\rangle &= \langle Px_1+Px,Px_2+Py\rangle\\
&=\langle Px_1,Px_2\rangle + \langle Px_1,Py\rangle+\langle Px,Px_2\rangle  + \langle Px,Py\rangle\\
&\ge\frac{1}{2}-\epsilon -2\epsilon-\epsilon^2\\
&\ge \frac{1}{2}-4\epsilon.
\end{align*}  The proof is complete.
\end{proof}

\begin{corollary}\label{R4}
Let $\{x_i\}_{i=1}^m$ be a unit-norm frame for $\RR^4$. If, for all $i\not= j$, $\|x_i-x_j\|<\epsilon<\frac 18$, then $\{x_i\}_{i=1}^m$
is not piecewise scalable with  any orthogonal projection of rank 2.
\end{corollary}

\begin{proof}
It is known that a frame for $\RR^2$ is not scalable if and only if all the frame vectors lie in the same open quadrant \cite{CX, KO}. Let $P$ be any rank 2 orthogonal projection on $\RR^4$. By Theorem \ref{thm proj}, 
either 
$  \left \langle Px_i, Px_j\right \rangle\ge \frac 12-4\epsilon>0$, 
or $\langle (I-P)x_i,\ (I-P)x_j  \rangle\ge \frac 12-\epsilon> 0$ for every   $i,j\in [m], i\not=j$. 
In either case, the projections  of the frame vectors lie in an open quadrant of $\R^2$ and cannot be scalable. 
By Theorem \ref{T-piecewise-1}, the frame is not piecewise scalable.
\end{proof}

We now generalize Corollary \ref{R4} in higher dimensional spaces. First, we need a lemma.

\begin{lemma}\label{lem2S4}
If $\frac 14\le\|x\|, \|y\|\leq 1$ and $\|x-y\|<\epsilon<\frac{1}{16}$, then 
\[\left\|\frac{x}{\|x\|}-\frac{y}{\|y\|}\right\|\leq 32\epsilon.\]
\end{lemma}
\begin{proof}
We compute: 
\begin{align*}
\left\|\frac{x}{\|x\|}-\frac{y}{\|y\|}\right\|&=\left\|\dfrac{\|y\|x-\|x\|y}{\|x\|\|y\|}\right\|\\
&\leq16(\|\|y\|x-\|y\|y+\|y\|y-y\|x\|\|)\\
&\leq 16(\|y\|\|x-y\|+\|y\|\|x-y\|)\\
&\leq 32 ||y||\,\|x-y\|<32\epsilon, 
\end{align*}
 which is the claim.
\end{proof}

\begin{theorem} Let $\F=\{x_i\}_{i=1}^m$ be a unit-norm frame for $\R^n$ such that $\|x_i-x_j\|<\epsilon<\frac{1}{64}$, for all $i\not=j$. Then $\F$ is not piecewise scalable with any projection of rank $k$, where $2\leq k\leq n-2$.
	\end{theorem}
\begin{proof} 
Re-indexing and switching $P$ and $(I-P)$ if necessary, by Theorem \ref{thm proj}, we may assume 
$ \langle Px_1,Px_i\rangle \ge \frac{1}{2}-4\epsilon$  for all $2\le i \le m$.
This implies $\|Px_i\|\geq \frac12-4\epsilon>\frac 14$ for all $i$. By Lemma \ref{lem2S4}, we have
\begin{align*}
(32\epsilon)^2>\left\|\dfrac{Px_i}{\|Px_i\|}-\dfrac{Px_j}{\|Px_j\|}\right\|^2=2-2\left\langle \dfrac{Px_i}{\|Px_i\|}, \dfrac{Px_j}{\|Px_j\|}\right\rangle,
\end{align*}
and hence,
\[\langle Px_i, Px_j\rangle \geq \left(1-\dfrac{(32\epsilon)^2}{2}\right)\|Px_i\|\|Px_j\|, \mbox { for all } i\not=j.\]

We now show that $\{Px_i\}_{i=1}^m$ is not scalable. We proceed by way of contradiction by assuming that there exist constants $\{a_i\}_{i=1}^m$ such that $\{a_iPx_i\}_{i=1}^m$ is a Parseval frame for $P(\RR^n)$. This implies that $\sum_{i=1}^m a_i^2\|Px_i\|^2=k \ge 2$.
 
Moveover, we have 
\begin{align*} 
\|Px_1\|^2&=a_1^2\|Px_1\|^4+\sum_{i=2}^ma_i^2|\langle Px_1, Px_i\rangle|^2\\
&\geq a_1^2\|Px_1\|^4+\sum_{i=2}^ma_i^2\left(1-\dfrac{(32\epsilon)^2}{2}\right)^2\|Px_1\|^2\|Px_i\|^2\\
&\geq \left(1-\dfrac{(32\epsilon)^2}{2}\right)^2a_1^2\|Px_1\|^4+\sum_{i=2}^ma_i^2\left(1-\dfrac{(32\epsilon)^2}{2}\right)^2\|Px_1\|^2\|Px_i\|^2\\
&\geq \left(1-\dfrac{(32\epsilon)^2}{2}\right)^2\|Px_1\|^2\sum_{i=1}^ma_i^2\|Px_i\|^2\\
&\geq\left(1-\dfrac{(32\epsilon)^2}{2}\right)^2\|Px_1\|^2\,k.
\end{align*}
It follows that $k=\rank P\leq \dfrac{1}{\left(1-\frac{(32\epsilon)^2}{2}\right)^2}<2$, which contradicts our assumptions. Thus, $\{Px_i\}_{i=1}^m$ is not scalable and  by Theorem \ref{T-piecewise-1}, $\F$ is not piecewise scalable with projection $P$.
\end{proof}

The following theorem gives a  sufficient conditions for a unit-norm frame in $\R^4$ to be piecewise scalable with  a projection $P$ of rank $2$. 

\begin{theorem}
Let $\F=\{x_i\}_{i=1}^m$ be a unit-norm frame for $\R^4$.  Suppose  that there exist linearly independent vectors $  x_1, x_2, x_3, x_4\in \F$  such that $\langle x_2, x_4\rangle=\langle x_3, x_4\rangle =0$, and $\langle x_1, x_4\rangle \not=0$. Then $\F$ is piecewise scalable.   
	\end{theorem}

\begin{proof}
We   construct an orthogonal projection $P$  of rank $2$ such that
\[Px_1 \perp Px_2 \mbox{ and } (I-P)x_3 \perp (I-P)x_4,\]
and all these vectors are nonzero.   By Corollary \ref{Cor2}, $\F$ is piecewise scalable with projection $P$.

Let $Q:\R^4\to \R^4$ be the orthogonal projection onto $\spn\{x_2, x_4\}$ and let $w=  (I- Q)x_1 $.   
Then, $w$ is orthogonal  to $x_2$ and $ x_4 $, and it is nonzero  because   the $x_i$'s are independent.   Furthermore,  $$ \langle x_1, w\rangle= \langle x_1, (I-Q) x_1\rangle=  \|(I-Q) x_1  \|^2= \|w\|^2 \ne 0.$$  
Thus, with  $u=\frac{w}{\|w\|} $ we have that $$\langle x_2, u\rangle=\langle  x_4, u\rangle=0, \mbox{ and } \langle x_1, u\rangle\ne 0.$$ Note also that $u\in \spn\{x_1, x_2, x_4\}$.
   
We  now let $R:\R^4\to\R^4$  be the orthogonal projection onto   $\spn\{x_1, x_3, u\}$   and let $z= (I-R)x_2$. So  $z$ is  orthogonal  to $x_1$,  $x_3$  and $u$. If it was $z=0$ then $x_2\in\spn\{x_1, x_3, u\}$, and  so $x_2=ax_1+bx_3+cu$ for some scalars $a, b, c$. Since $u\in \spn\{x_1, x_2, x_4\}$, we have that $u=\alpha x_1+\beta x_2+\gamma x_4$ for some $\alpha, \beta, \gamma$. 
It follows that 
\[(a+c\alpha) x_1+(c\beta-1)x_2+bx_3+c\gamma x_4=0.\]
Since $x_i$'s are linearly independent, we must have 
\[a+c\alpha=0, \ c\beta =1, \ b=0, \ c\gamma =0, \] and so $b=\gamma=0$. Thus, $u=\alpha x_1+\beta x_2$. Note that $\alpha\not=0$, since $u\perp x_2.$ But $x_4$ is orthogonal to $x_2$ and $u$, and so also  $\langle x_4, x_1\rangle =0$, a contradiction. Thus,  $z\not=0$, and 
$$ \langle x_2, z\rangle= \langle x_2, (I-R) x_2 \rangle=  \|(I-R) x_2  \|^2= \|z\|^2 \ne 0.$$  

With  $v= \frac{z}{\|z\|}$,   we have that  $$\langle x_1, v\rangle= \langle x_3, v\rangle =\langle u, v\rangle =0, \mbox { and } \langle x_2, v\rangle \not=0.$$  

Now let $P$ be the orthogonal projection onto $\spn\{u, v\}$. Then 
\[Px=\langle x, u\rangle u+\langle x, v\rangle v \ \mbox{ for } x\in \RR^4.\]
We can see at once  that
\begin{align*}
\langle Px_1, Px_2\rangle =\langle x_1, u\rangle \langle x_2, u\rangle +\langle x_1, v\rangle\langle x_2, v\rangle =0,
\end{align*}
\[\langle (I-P)x_3, (I-P)x_4\rangle =\langle x_3, x_4\rangle-\langle x_3, u\rangle \langle x_4, u\rangle -\langle x_3, v\rangle\langle x_4, v\rangle =0.\]

 To complete the proof, we need to show that $Px_1, Px_2, (I-P)x_3$, and $(I-P)x_4$ are nonzero. Clearly, $Px_1, Px_2\not=0$ since $\langle x_1, u\rangle \not=0$ and $\langle x_2, v\rangle \not=0$. Let us show that $(I-P)x_3$ and $(I-P)x_4$ are nonzero.
 
Indeed, if $x_3=Px_3=\langle x_3, u\rangle u$, then $x_3\in \spn\{x_1, x_2, x_4\}$, which is impossible. Similary, if $x_4=Px_4$, then $x_4=\langle x_4, v\rangle v.$ Since $\langle x_1, v\rangle=0$, it follows that $\langle x_1, x_4\rangle =0$. This contradicts our assumptions.
 \end{proof}

\section{ Properties of scaling constants of piecewise scalable frames}
In this section, we give some basic properties of scaling constants of piecewise scalabe frames. Let us start by showing that we cannot scale vectors of unit-norm frames by large numbers to get Parseval frames.

\begin{proposition}\label{constants}
	Let $\{x_i\}_{i=1}^m$ be a unit-norm frame for $\mathbb{R}^n$. If $\{c_ix_i\}_{i=1}^m$ is a  Parseval  frame for some nonzero scalars $\{c_i\}_{i=1}^m$, then 
	\[\max_{i}c_i^2\leq 1=\frac{\sum_{i=1}^{m}c_i^2}{n}.\]
Moreover,  $|c_i|=1$ if and only if $x_i\perp x_j$
for all $1\le j\not= i \le m$.	
\end{proposition}

	\begin{proof}
		For any $i\in [m]$, we have that
		\[\|c_ix_i\|^4\leq \sum_{j=1}^{m}|\langle c_ix_i, c_jx_j\rangle|^2= \|c_ix_i\|^2.\]
		It follows that $c_i^2\leq 1 $ for all $i$. Since this is a Parseval frame, we have
\[ \sum_{i=1}^m\|c_ix_i\|^2= \sum_{i=1}^mc_i^2=n.\] The second equality follows.

	To prove the second part of the proposition, we assume that $|c_i|=1$ for some $i$. Then
		\[ \|x_i\|^2=\sum_{j=1}^m|\langle x_i,c_jx_j\rangle|^2= |\langle x_i,x_i\rangle|^2 + \sum_{i\not= j =1}^m|\langle x_i,c_jx_j\rangle|^2.\]
		So $|c_j||\langle x_i,x_j\rangle |=|\langle x_i,c_jx_j\rangle|=0$ for all $1\le j\not= i \le 1$. 
		
		The converse is obvious.
	\end{proof}

 The following propositions give a bound for the scaling constants of a piecewise  scalable frame.
\begin{proposition}
	Let $\F=\{x_i\}_{i=1}^m$ be a unit-norm frame for $\RR^n$. If $\F$ is piecewise scalable with constants $\{a_i, b_i\}_{i=1}^m$, then
	\[n\leq \sum_{i=1}^{m}\max(a^2_i, b^2_i).\]
\end{proposition}
\begin{proof} Let $P$ be an orthogonal projection for which $\{a_iPx_i+b_i(I-P)x_i\}_{i=1}^m$ is a Parseval frame for $\RR^n$.
	For any $i\in [m]$, we have
	\begin{align*}
		\|a_iPx_i+b_i(I-P)x_i\|^2&=\|a_iPx_i\|^2+\|b_i(I-P)x_i\|^2\\
		&\leq \max(a_i^2, b_i^2)(\|Px_i\|^2+(I-P)x_i)\|^2\\
		&=\max(a_i^2, b_i^2).
	\end{align*}
	From $\displaystyle \sum_{i=1}^{m}\|a_iPx_i+b_i(I-P)x_i\|^2=n$, we get the claim.
\end{proof}
{ 
	\begin{proposition}
		Let $\{x_i\}_{i=1}^m$ be a unit-norm frame for $\RR^n$. Let   $P$ be an orthogonal projection. If $\{a_iPx_i\}_{i=1}^m$ and $\{b_i(I-P)x_i\}_{i=1}^m$ are   Parseval frames for $P(\mathbb{R}^n)$ and $(I-P)(\mathbb{R}^n)$ respectively, then 
		\[|a_ib_i|\leq \sqrt{a_i^2+b_i^2} \text{ and } \min(|a_i|, |b_i|)\leq \sqrt{2} \text{ for all } i\in [m].\]
	\end{proposition}
	\begin{proof}
		Since $\{a_iPx_i\}_{i=1}^m$ is a Parseval frame, for any $i\in [m]$, we have
		\[\|a_iPx_i\|^4\leq \sum_{j=1}^{m}|\langle a_iPx_i, a_jPx_j\rangle|^2=\|a_iPx_i\|^2. \]
		It follows that
		\[a^2_ib^2_i\|Px_i\|^2\leq b_i^2 \text{ for all } i.\]
		Similarly, 
		\[a_i^2b_i^2\|(I-P)x_i\|^2\leq a_i^2\ \text{ for al } i.\]
		Since $\|Px_i\|^2+\|(I-P)x_i\|^2=\|x_i\|^2=1$, the first conclusion follows.
		
		To see the second conclusion, note that 
		\[a_i^2\|Px_i\|^2\leq 1 \text{ and } b_i^2\|(I-P)x_i\|^2\leq 1 \text{ for all } i.\]
		If $|a_i|\leq |b_i|$ then 
		$$a_i^2=a_i^2(\|Px_i\|^2+(I-P)x_i\|^2)\leq a_i^2\|Px_i\|^2+b_i^2\|(I-P)x_i\|^2\leq 2.$$
		Similarly, if $|b_i|\leq |a_i|$, then $b_i^2\leq 2$. Thus, $\min(|a_i|, |b_i|)\leq \sqrt{2}$ for $i\in [m]$.
	\end{proof}
	
	\medskip
	Unlike the properties of scaling constants of scalable frames as shown in Proposition \ref{constants},  for piecewise scalable frames, the constants can be arbitrarily large. The reason is that the vectors $Px_i $ or $(I-P)x_i $ might have tiny norms.
	
	\begin{example}\label{E412} Let $\epsilon>0$ and let $\{x_i\}_{i=1}^2$ be a unit-norm frame for $\mathbb{R}^2$, where 
		\[x_1=(\epsilon, \sqrt{1-\epsilon^2}), \quad x_2=(\sqrt{2}/2, \sqrt{2}/2). \]
		Let $P$ be the orthogonal projection onto $\spn \{e_1\}$. Then 
		\[Px_1=(\epsilon,0), \quad (I-P)x_1=(0,\sqrt{1-\epsilon^2}),\]  \[Px_2=(\sqrt{2}/2, 0), \quad (I-P)x_2=(0,\sqrt{2}/2).\]
		We can see that $\{a_iPx_i+b_i(I-P)x_i\}_{i=1}^2$ is a Parseval frame with $a_1=1/\epsilon, b_1=0, a_2=0, b_2=\sqrt{2}.$
	\end{example}

\end{document}